\documentclass{amsart}

\usepackage[utf8]{inputenc}
\usepackage{amsmath,amsfonts,amsthm,amssymb,verbatim,etoolbox,color}
\usepackage{hyperref}

\usepackage[most]{tcolorbox}
\newcommand{\ind}[1]{1_{#1}}
\newcommand{\lo}[1]{\mathcal{L}(#1)}

\newcommand{\bbf}{\mathbb{F}}

\newcommand{\bbr}{\mathbb{R}}

\newcommand{\bbc}{\mathbb{C}}

\newcommand{\abs}[1]{\left\lvert #1\right\rvert}
\newcommand{\Abs}[1]{\lvert #1\rvert}
\newcommand{\brac}[1]{\left( #1\right)}

\newcommand{\norm}[1]{\left\lVert #1\right\rVert}

\newcommand*{\bbe}{
  \mathop{
    \mathchoice{\vcenter{\hbox{\larger[4]$\mathbb{E}$}}}
               {\kern0pt\mathbb{E}}
               {\kern0pt\mathbb{E}}
               {\kern0pt\mathbb{E}}
  }\displaylimits
}
\DeclareMathOperator{\rk}{rk}

\renewcommand{\subset}{\subseteq}

\newtheorem{theorem}{Theorem}
\newtheorem{lemma}[theorem]{Lemma}

\newtheorem{proposition}[theorem]{Proposition}
\theoremstyle{definition}

\title[An improvement to the Kelley-Meka bounds]{An improvement to the Kelley-Meka bounds on three-term arithmetic progressions}
\author{Thomas F. Bloom and Olof Sisask}
\begin{document}

\maketitle
\begin{abstract}
In a recent breakthrough Kelley and Meka proved a quasipolynomial upper bound for the density of sets of integers without non-trivial three-term arithmetic progressions. We present a simple modification to their method that strengthens their conclusion, in particular proving that if $A\subset\{1,\ldots,N\}$  has no non-trivial three-term arithmetic progressions then
\[\lvert A\rvert \leq \exp(-c(\log N)^{1/9})N\]
for some $c>0$.
\end{abstract}

The question of how large a subset of $\{1,\ldots,N\}$ without three-term arithmetic progressions can be is one of the most central in additive combinatorics. Recently Kelley and Meka \cite{KM} achieved a breakthrough new bound, proving that such a set must have size at most
\[ \exp(-c(\log N)^{1/12})N \]
for some constant $c>0$. For contrast, it is known by a result of Behrend \cite{Be} that there are such sets of size at least $\exp(-c'(\log N)^{1/2}))N$ for some constant $c'>0$ (with small improvements by Elkin \cite{El} and Green and Wolf \cite{GW}, which do not change the bound's essential shape). The bound achieved by Kelley and Meka is a dramatic improvement over any bounds previously available (for example in \cite{BSa}), which were all of the shape $N/(\log N)^{O(1)}$.

In this note we observe that a small modification to Kelley and Meka's argument (more precisely in the application of almost-periodicity) yields a slight quantitative improvement. 

\begin{theorem}\label{th-main-int}
If $A\subseteq \{1,\ldots,N\}$ contains only trivial three-term arithmetic progressions, then
\[\abs{A} \leq \exp(-c(\log N)^{1/9})N\]
for some constant $c>0$. 
\end{theorem}

A more elaborate version of the idea in this note allows for further improvement of the exponent to $5/41$ (see the remarks after the proof of Lemma~\ref{lemma:symmetry_subspace}), but the necessary technical overheads obscure the essential idea. Since we expect other ideas to render such lengthy technical optimisation redundant anyway, in this note we just present the relatively clean modification that allows for $1/9$.

Similar quantitative improvements are available for other applications of Kelley and Meka's method. For example, the new argument in the `model setting' of $\bbf_q^n$ yields the following. 

\begin{theorem}\label{th-main-ff}
If $q$ is an odd prime and $A\subseteq \bbf_q^n$ has no non-trivial three-term progressions, then 
\[\abs{A} \ll q^{n-cn^{1/7}}\]
for some constant $c>0$.
\end{theorem}
Kelley and Meka \cite{KM} proved a similar bound with $1/9$ in place of $1/7$. For this problem much better bounds (of the shape $q^{(1-c)n}$) were proved by Ellenberg and Gijswijt \cite{EG} using the polynomial method of Croot, Lev, and Pach \cite{CLP}. As usual in this area, however, $\bbf_q^n$ is useful as a simpler setting than $\{1,\ldots,N\}$ that still displays most of the important ideas.

Another application of the (quantitatively improved) Kelley-Meka argument yields the following. Even in the model setting of $\bbf_q^n$, this type of result does not follow from the polynomial method.
\begin{theorem}\label{th-ff}
If $A\subseteq \bbf_q^n$ has density $\alpha=\abs{A}/q^n$ and $\gamma\in(0,1]$ then there is some affine subspace $V\leq \bbf_q^n$ of codimension\footnote{We recall our notational convention from \cite{BS} that $\lo{\delta}=\log(2/\delta)$ when $\delta\in(0,1]$.} $O(\lo{\alpha}^5\lo{\gamma}^2)$ such that 
\[\abs{(A+A)\cap V}\geq (1-\gamma)\abs{V}.\]
\end{theorem}

For comparison, Kelley-Meka \cite{KM} proved this with codimension $O(\lo{\alpha}^5\lo{\gamma}^4)$, and Sanders \cite{Sa2} proved this with codimension $O(\lo{\alpha}^4\gamma^{-2})$.

There is a recent application of our improvement to the Kelley-Meka machinery, in the setting of $\mathbb{F}_2^n$, to Ramsey theory: Hunter and Pohoata \cite{HP23} use essentially Theorem~\ref{th-ff} to improve the known bounds for the Ramsey problem of finding monochromatic subspaces in $2$-colourings of the $1$-dimensional subspaces of $\bbf_2^n$.

Finally, we present a new bound for long arithmetic progressions in $A+A+A$.
\begin{theorem}\label{th-3A}
If $A\subseteq \{1,\ldots,N\}$ has size $\alpha N$ then $A+A+A$ contains an arithmetic progression of length at least
\[ \exp(-O(\lo{\alpha}^2))N^{\Omega(1/\lo{\alpha}^{7})}. \]
\end{theorem}
 The authors proved a weaker version of this, with $9$ in place of $7$, in \cite{BS} as an application of a technically `smoothed' version of the Kelley-Meka method. A construction due to Freiman, Halberstam, and Ruzsa \cite{FHR} shows that no exponent better than $\Omega(1/\lo{\alpha})$ is possible.

Since our new contribution is only a small modification of the argument of Kelley and Meka, we will be relatively brief, and just describe the changes required. In particular we assume that the reader is familiar with the  simplified form of the Kelley-Meka argument as presented in \cite{BS}. We will use the same notation and conventions as given in \cite[Section 2]{BS}, which we briefly recall below for the convenience of the reader.

In Section~\ref{sec-boot} we present the novel contribution of this paper, a quantitatively improved bootstrapping of almost-periodicity. In Section~\ref{sec-sumup} we explain how this improved almost-periodicity should be inserted into the Kelley-Meka argument (in the form presented in \cite{BS}) to prove our main results.

\subsection*{Acknowledgements} The first author is supported by a Royal Society University Research Fellowship. 

\subsection*{Notational conventions}
Logarithmic factors will appear often, and so in this paper we use the convenient abbreviation $\lo{\alpha}$ to denote $\log(2/\alpha)$. In statements which refer to $G$, this can be taken to be any finite abelian group (although for the applications this will always be either $\mathbb{F}_q^n$ or $\mathbb{Z}/N\mathbb{Z}$). We use the normalised counting measure on $G$, so that 
\[\langle f,g\rangle = \bbe_{x\in G}f(x)\overline{g(x)}\textrm{ and }\norm{f}_p=\brac{\bbe_{x\in G}\abs{f(x)}^p}^{1/p}\textrm{ for }1\leq p<\infty,\]
where $\bbe_{x\in G}=\frac{1}{\lvert G\rvert}\sum_{x\in G}$. For any $f,g:G\to \bbc$ we define the convolution  and the difference convolution\footnote{We caution that, while convolution is commutative and associative, difference convolution is in general neither.} as
\[f\ast g(x)=\bbe_y f(y)g(x-y)\quad\textrm{and}\quad f\circ g(x)=\bbe_yf(x+y)\overline{g(y)}.\]

For some purposes it is conceptually cleaner to work relative to other non-negative functions on $G$, so that if $\mu:G\to\bbr_{\geq 0}$ has $\norm{\mu}_1=1$ we write 
\[\norm{f}_{p(\mu)}=\brac{\bbe_{x\in G}\mu(x)\abs{f(x)}^p}^{1/p}\textrm{ for }1\leq p<\infty.\]
(The special case above is the case when $\mu\equiv 1$.) We write $\mu_A=\alpha^{-1}\ind{A}$ for the normalised indicator function of $A$ (so that $\norm{\mu_A}_1=1$). We will sometimes speak of $A\subseteq B$ with relative density $\alpha=\abs{A}/\abs{B}$.

The Fourier transform of $f:G\to\bbr$ is $\widehat{f}:\widehat{G}\to \bbc$ defined for $\gamma\in\widehat{G}$ as
\[\widehat{f}(\gamma)=\bbe_{x\in G}f(x)\overline{\gamma(x)},\]
where $\widehat{G} = \{ \gamma : G \to \bbc^\times : \text{$\gamma$ a homomorphism} \}$ is the dual group of $G$.

Finally, we use the Vinogradov notation $X \ll Y$ to mean $X = O(Y)$, that is, there exists some constant $C>0$ such that $\abs{X}\leq CY$. We write $X\asymp Y$ to mean $X\ll Y$ and $Y\ll X$. The appearance of parameters as subscripts indicates that this constant may depend on these parameters (in some unspecified fashion).

\section{An improved bootstrapping procedure}\label{sec-boot}
The new contribution of this paper is to note that the `bootstrapping procedure', in which a set of almost-periods is converted into a subspace (or, more generally, a Bohr set) of almost-periods, can be made more efficient, at least in the applications relevant to the Kelley-Meka argument.

\subsection{The $\bbf_q^n$ case}
We will first present the new idea in the technically simpler model case of $\bbf_q^n$. For this we will use the following form of almost-periodicity, a special case of \cite[Theorem 3.2]{SS}, which is sufficient for our purposes.

\begin{theorem}[$L^\infty$ almost-periodicity]\label{th-liap}
Let $\epsilon>0$ and $k\geq 1$. Let $S\subseteq G$ and $A_1,A_2\subseteq G$ with densities $\alpha_1,\alpha_2$ respectively. There is a set $X\subseteq G$ of size
\[\abs{X}\gg \exp(-O(\epsilon^{-2}k^2\lo{\alpha_1}\lo{\alpha_2}))\abs{G}\]
such that
\[\|\mu_X^{(k)}\ast\mu_{A_1}\circ\mu_{A_2}\ast \ind{S}-\mu_{A_1}\circ\mu_{A_2}\ast \ind{S}\|_{\infty}\leq \epsilon. \]
\end{theorem}

Bootstrapping refers to the process where the almost-period factor of $\mu_X^{(k)}$ is replaced by a more algebraically structured factor of $\mu_V$, where $V$ is a subspace. This is achieved by passing to Fourier space and considering the subspace of elements which annihilate (or approximately annihilate) those characters where $\Abs{\widehat{\mu_X}}$ is large. The problem is that to control the error term in such a replacement we need to `cancel out' the quantity 
\[\sum_\gamma \Abs{\widehat{\mu_{A_1}}(\gamma)}\Abs{\widehat{\mu_{A_2}}(\gamma)}\Abs{\widehat{\ind{S}}(\gamma)}.\]
Using the trivial bound $\Abs{\widehat{\mu_{A_2}}(\gamma)}\leq 1$, the Cauchy-Schwarz inequality, and Parseval's identity, we can bound this above by
\begin{align*}
\sum_\gamma \Abs{\widehat{\mu_{A_1}}(\gamma)}\Abs{\widehat{\ind{S}}(\gamma)}
&\leq \brac{\sum_\gamma \Abs{\widehat{\mu_{A_1}}(\gamma)}^2}^{1/2}\brac{\sum_\gamma \Abs{\widehat{\ind{S}}(\gamma)}^2}^{1/2}\\
&=\alpha_1^{-1/2}\mu(S)^{1/2}\\
&\leq \alpha_1^{-1/2}.
\end{align*}
This is multiplied by a factor of $\abs{\widehat{\mu_X}(\gamma)}^k$, which we can take to be $\leq 2^{-k}$ (since we can discard the contribution from those $\gamma$ with $\abs{\widehat{\mu_X}(\gamma)}\geq 1/2$ by passing a subspace of small codimension). In particular to `cancel out' the contribution from this sum we need to take $k\approx \lo{\alpha_1}$. For many applications of almost-periodicity, when $S$ is an arbitrary set, this is the best that we can do.

In the Kelley-Meka application, however, $S$ is a structured set, and we can exploit that here. Essentially, we know that $\mu_A\circ \mu_A$ is large pointwise on $S$ (for some set $A$ which is denser than both $A_1$ and $A_2$), and therefore at a crucial stage of the argument we can replace $\ind{S}$ by $\mu_A\circ \mu_A$ before bootstrapping. This leads to the use of the alternative bound
\[\sum_\gamma \Abs{\widehat{\mu_{A_1}}(\gamma)}\Abs{\widehat{\mu_{A_2}}(\gamma)}\Abs{\widehat{\mu_A}(\gamma)}^2\leq \alpha^{-1}.\]
Thus we have a $\lo{\alpha}$ term in place of a $\lo{\alpha_1}$ term, and since $\lo{\alpha_1}\approx \lo{\alpha}^2$ in the Kelley-Meka method this leads to an improvement in the final bounds. 

The following lemma and proof is a precise statement of the above idea suitable for our applications.

\begin{lemma}\label{lemma:symmetry_subspace}
Let $\epsilon\in(0,1/8)$. Let $S\subseteq \bbf_q^n$ and $A,A_1,A_2\subseteq \bbf_q^n$ be sets with densities $\alpha,\alpha_1,\alpha_2$ respectively, such that
\begin{enumerate}
\item $\langle \mu_{A_1}\circ \mu_{A_2},\ind{S}\rangle \geq 1-\epsilon$ and
\item $\mu_A\circ \mu_A(x) \geq 1+4\epsilon$ for any $x\in S$.
\end{enumerate}
There exists a subspace $V\leq \bbf_q^n$ of codimension
\[\ll_\epsilon \lo{\alpha}^2\lo{\alpha_1}\lo{\alpha_2}\]
such that $\norm{\mu_V\ast \mu_A}_\infty \geq 1+\epsilon/2$.
\end{lemma}
\begin{proof}
Let $k\geq 2$ be chosen later and $X$ be as in Theorem~\ref{th-liap}, so that
\[\langle \mu_X^{(k)}\ast \mu_{A_1}\circ \mu_{A_2},\ind{S}\rangle \geq 1-2\epsilon\]
and
\[\abs{X} \gg \exp(-O_\epsilon(k^2\lo{\alpha_1}\lo{\alpha_2}))\Abs{\bbf_q^n}.\]
It follows that 
\[\langle \mu_X^{(k)}\ast \mu_{A_1}\circ \mu_{A_2},\mu_A\circ \mu_A\rangle \geq (1+4\epsilon)(1-2\epsilon)\geq 1+\epsilon.\]
Let $V\leq \bbf_q^n$ be the subspace orthogonal to all those characters in 
\[\Delta_{1/2}(X)=\{\gamma : \Abs{\widehat{\mu_X}(\gamma)}\geq 1/2\}.\]
By Chang's lemma (as given in \cite[Lemma 4.36]{TV}, for example), $V$ has codimension 
\[\ll \log(\Abs{\bbf_q^n}/\abs{X})\ll_\epsilon k^2\lo{\alpha_1}\lo{\alpha_2}.\]
Furthermore, if we let $F=\mu_{A_1}\circ \mu_{A_2}\ast \mu_A\circ \mu_A$ for brevity, for all $t\in V$ we have 
\begin{align*}
\| \tau_t(\mu_X^{(k)}\ast F)-\mu_X^{(k)}\ast F\|_\infty
&\leq \sum_\gamma \Abs{\widehat{\mu_X}(\gamma)}^k\Abs{\widehat{F}(\gamma)}\abs{\gamma(t)-1}\\
&\leq 2\sum_{\gamma\not\in \Delta_\eta(X)}
\Abs{\widehat{\mu_X}(\gamma)}^k\Abs{\widehat{F}(\gamma)}\\
&\leq 2^{1-k}\sum_\gamma \Abs{\widehat{F}(\gamma)}.
\end{align*}
We now note that
\[\sum_\gamma \Abs{\widehat{F}(\gamma)}\leq \sum_\gamma \Abs{\widehat{\mu_A}(\gamma)}^2 \leq \alpha^{-1}.\]
In particular, we can choose $k\ll_\epsilon \lo{\alpha}$ so that, for any $t \in V$,
\[\| \tau_t(\mu_X^{(k)}\ast F)-\mu_X^{(k)}\ast F\|_\infty\leq \epsilon/2.\]
It follows that
\[\langle \mu_V\ast \mu_X^{(k)}\ast \mu_{A_1}\circ \mu_{A_2},\mu_A\circ \mu_A\rangle \geq 1+\epsilon/2,\]
whence $\norm{\mu_V\ast \mu_A}_\infty \geq 1+\epsilon/2$ as required.
\end{proof}

Note that in this proof we used a relatively trivial bound of
\[\sum_\gamma \abs{\widehat{\mu_{A_1}}(\gamma)\widehat{\mu_{A_2}}(\gamma)\widehat{\mu_{A_1}}(\gamma)}^2\leq \sum_\gamma \abs{\widehat{\mu_A}(\gamma)}^2.\]
We could instead retain the $\mu_{A_i}$ factors, resulting in an upper bound of
\[\langle \mu_{A_1}\circ \mu_{A_1},\mu_A\circ \mu_A\rangle^{1/2}\langle \mu_{A_2}\circ \mu_{A_2},\mu_A\circ \mu_A\rangle^{1/2}.\]
In particular, if both of these inner products are small (e.g. $\ll \lo{\alpha}^{O(1)}$) then we could attain a sharper form of Lemma~\ref{lemma:symmetry_subspace}, with $k\approx \log \lo{\alpha}$. If not, say
\[\langle \mu_{A_1}\circ \mu_{A_1},\mu_A\circ \mu_A\rangle\geq \lo{\alpha},\]
then this is a large discrepancy over the `expected value' of this inner product, which is $1$. This in turn can be fed back into the Kelley-Meka machinery to produce another density increment. This is not an immediate win, since the density of $A_1$ is much smaller than that of $A$, so it is not clear that we have gained more than we lost. Nonetheless a small improvement can be attained this way, optimising carefully, but this requires taking apart the Kelley-Meka machinery and a technical lengthy detour. Again, since we expect future ideas to make the gains from such an optimisation redundant anyway, we have chosen to present only the simpler version.

Nonetheless, the possibility of improved bounds should be kept in mind, and the reader interested in applying an improved bootstrapping similar to Lemma~\ref{lemma:symmetry_subspace} to other problems should explore whether a good upper bound on something like
\[\langle \mu_{A_1}\circ \mu_{A_1},\mu_A\circ \mu_A\rangle\]
is available in their application.
\subsection{The general case}

We now present the general case of the improved bootstrapping procedure described in the previous subsection, required for the integer case. We will assume that the reader is familiar with the vocabulary and basic properties of Bohr sets (see, for example, \cite[Appendix 1]{BS}). In this section $G$ denotes any finite abelian group.

We will use the following more general form of almost-periodicity, which is proved as \cite[Theorem 5.1]{SS}.

\begin{theorem}[$L^\infty$ almost-periodicity]\label{th-liap-gen}
Let $\epsilon>0$ and $k,K\geq 2$. Let $A_1,A_2,S,B\subseteq G$ and $\abs{A_2+B}\leq K\abs{A_2}$. Let $\eta=\abs{A_1}/\abs{S}$. There is a set $X\subseteq B$ of size
\[\abs{X}\gg \exp(-O(\epsilon^{-2}k^2\lo{\eta}\log K))\abs{B}\]
such that
\[\|\mu_X^{(k)}\ast\mu_{A_1}\circ\mu_{A_2}\ast \ind{S}-\mu_{A_1}\circ\mu_{A_2}\ast \ind{S}\|_{\infty}\leq \epsilon. \]
\end{theorem}

The more general form of Lemma \ref{lemma:symmetry_subspace}, required for the application to the integers, is more complicated in technicalities only. It is important to note, however, that there is an additional loss in the size of the Bohr set comparable to $d\lo{\alpha}$ (where $d$ is the rank of the Bohr set) -- it is ultimately this which is responsible for `losing two logs' between the $\bbf_p^n$ and the integer case.

\begin{lemma}\label{lemma-genimp}
There is a constant $c>0$ such that the following holds. Let $\epsilon\in (0,1/10)$ and $B,B',B''\subseteq G$ be regular Bohr sets of rank $d$. Suppose that $A\subseteq B$, $A_1\subseteq B'$, and $A_2\subseteq B''-x$ (for some $x$) with densities $\alpha,\alpha_1,\alpha_2$ respectively. Let $S$ be any set with $\abs{S}\leq 2\abs{B'}$ such that 

\begin{enumerate}
\item $\langle \mu_{A_1}\circ \mu_{A_2},\ind{S}\rangle \geq 1-\epsilon$ and
\item $\mu_A\circ \mu_A(x) \geq (1+2\epsilon)\mu(B)^{-1}$ for any $x\in S$.
\end{enumerate}
Let $L=\lo{\alpha/d\lo{\alpha_1}\lo{\alpha_2}}$. There is a regular Bohr set $B'''\subseteq B''$ of rank at most

\[\leq d+O_\epsilon(\lo{\alpha}^2\lo{\alpha_1}\lo{\alpha_2})\]
and 
\[\abs{B'''}\geq \exp(-O_\epsilon(L(d +\lo{\alpha}^2\lo{\alpha_1}\lo{\alpha_2})))\abs{B''}\]
such that $\norm{\mu_{B'''}\ast \mu_A}_\infty \geq (1+\epsilon/4)\mu(B)^{-1}$. 
\end{lemma}
Note that in our application we have $\alpha_1,\alpha_2 \geq \exp(-O(\lo{\alpha}^2))$ and $d\leq \alpha^{-O(1)}$, and hence the parameter $L$ is $O(\lo{\alpha})$.
\begin{proof}
Let $k\geq 2$ be chosen later and $X$ be as in Theorem~\ref{th-liap-gen}, applied with $B$ replaced by $B''_\rho$, where $\rho=c/100d$ for a constant $c\in (1/2,1)$ chosen so that $B''_\rho$ is regular. By regularity of $B''$ 
\[\abs{A_2+B''_\rho}\leq \abs{B''+B''_\rho}\leq 2\abs{B''}\leq 2\alpha_2^{-1}\abs{A_2},\]
and so we can take $K=2\alpha_2^{-1}$ in Theorem~\ref{th-liap-gen}. We also have $\eta=\abs{A_1}/\abs{S}\geq \alpha_1/2$. We can thus find some $X\subseteq B_{\rho}''$ such that 
\[\langle \mu_X^{(k)}\ast \mu_{A_1}\circ \mu_{A_2},\ind{S}\rangle \geq 1-\tfrac{5}{4}\epsilon\]
and
\[\abs{X}\gg \exp(-O_\epsilon(k^2\lo{\alpha_1}\lo{\alpha_2}))\Abs{B''_\rho}.\]
It follows that
\[\langle \mu_X^{(k)}\ast \mu_{A_1}\circ \mu_{A_2},\mu_A\circ \mu_A\rangle \geq (1+\epsilon/2)\mu(B)^{-1}.\]
By Chang's lemma (for example as given in \cite[Proposition 5.3]{SS}) there is a regular Bohr set $B'''\subseteq B''_{\rho}$ of rank 
\[\leq d+O_\epsilon(k^2\lo{\alpha_1}\lo{\alpha_2})\]
and 
\[\abs{B'''}\geq \exp(-O_\epsilon(L(d +k^2\lo{\alpha_1}\lo{\alpha_2})))\abs{B''}\]
such that $\abs{\gamma(t)-1}\leq \epsilon\alpha/10$ for all $\gamma\in \Delta_{1/2}(X)$ and $t \in B'''$. Writing $F=\mu_{A_1}\circ \mu_{A_2}\ast \mu_A\circ \mu_A$ for brevity, it follows that for all $t\in B'''$ we have 
\begin{align*}
\| \tau_t(\mu_X^{(k)}\ast F)-\mu_X^{(k)}\ast F\|_\infty
&\leq \sum_\gamma \Abs{\widehat{\mu_X}(\gamma)}^k\Abs{\widehat{F}(\gamma)}\abs{\gamma(t)-1}\\
&\leq (\epsilon\alpha/10+2^{1-k})\sum_\gamma \Abs{\widehat{F}(\gamma)}.
\end{align*}
By the Cauchy-Schwarz inequality
\[\sum_\gamma \Abs{\widehat{F}(\gamma)}\leq \sum_\gamma \Abs{\widehat{\mu_A}(\gamma)}^2\leq \alpha^{-1}\mu(B)^{-1}.\]
In particular, we choose $k\ll_\epsilon \lo{\alpha}$ so that, for each $t \in B'''$
\[\| \tau_t(\mu_X^{(k)}\ast F)-\mu_X^{(k)}\ast F\|_\infty\leq \tfrac{1}{4}\epsilon \mu(B)^{-1}.\]
It follows that
\[\langle \mu_{B'''}\ast \mu_X^{(k)}\ast \mu_{A_1}\circ \mu_{A_2},\mu_A\circ \mu_A\rangle \geq (1+\epsilon/4)\mu(B)^{-1},\]
whence $\norm{\mu_{B'''}\ast \mu_A}_\infty \geq (1+\epsilon/4)\mu(B)^{-1}$ as required.
\end{proof}

\section{Modifying the Kelley-Meka argument}\label{sec-sumup}

\subsection{The $\bbf_q^n$ case}

Both Theorems~\ref{th-main-ff} and \ref{th-ff} follow by an iterative application of the following quantitative improvement of \cite[Proposition 12]{BS}.

\begin{proposition}\label{th-ip}
Let $q$ be any prime and $n\geq 1$. If $A,C\subseteq \bbf_q^n$, where $A$ has density $\alpha$ and $C$ has density $\gamma$, then for any $\epsilon\in(0,1)$, either
\begin{enumerate}
\item $\abs{\langle \mu_A\ast \mu_A,\mu_C\rangle -1}\leq \epsilon$ or
\item there is a subspace $V$ of codimension 
\[\ll_\epsilon \lo{\alpha}^4\lo{\gamma}^2\]
such that $\norm{\ind{A}\ast \mu_V}_\infty \geq (1+\epsilon/64)\alpha$.
\end{enumerate}
\end{proposition}
\begin{proof}
By the argument of Kelley and Meka (such as a combination of \cite[Lemma 7, Corollary 9, Lemma 11]{BS}, as described in the proof of \cite[Proposition 13]{BS}) if the first alternative fails then there are sets $A_1,A_2$, both of density
\[\geq \exp(-O_\epsilon(\lo{\alpha}\lo{\gamma})),\]
such that $\langle \mu_{A_1}\circ \mu_{A_2},\ind{S}\rangle\geq 1-\epsilon/32$ where $S=\{ x : \mu_A\circ \mu_A(x)\geq 1+\epsilon/8\}$. The result now follows from Lemma~\ref{lemma:symmetry_subspace}.
\end{proof}

\subsection{The general case}

Similarly, Theorems~\ref{th-main-int} and \ref{th-3A} follow from the following quantitative improvement of \cite[Proposition 14]{BS}. The deduction in this case is less routine, but is unchanged from the argument in \cite{BS}, so we will not reproduce the details here.

To summarise, however, the method of Kelley and Meka allows one to show that if there are too few three-term arithmetic progressions in $A\subseteq B$ then the hypothesis of the below holds with $\epsilon \gg 1$ and $p\asymp \lo{\alpha}$. The density increment in the conclusion can hold only $\lo{\alpha}$ many times. Beginning with the trivial rank $0$ Bohr set and iterating, therefore, we arrive at some Bohr set $B$ with rank $d\ll \lo{\alpha}^7$ and density $\abs{B}\gg \exp(-O(\lo{\alpha}^9))\abs{G}$ on which we have the `expected' number of three-term arithmetic progressions. That is, there is some $A'\subseteq (A-x)\cap B$ for some $x$ which has $\gg \abs{B}^2$ many arithmetic progressions. By assumption $A'$ only contains trivial three-term arithmetic progressions, and so this forces $\abs{B}\ll 1$. Rearranging and using our lower bound on $\abs{B}$ this implies $\alpha \leq \exp(-c(\log \abs{G})^{1/9})$ for some $c>0$ as required. 
\begin{proposition}\label{prop-it}
There is a constant $c>0$ such that the following holds. Let $\epsilon>0$ and $p,k\geq 1$ be integers such that $(k,\abs{G})=1$ and $p\leq \alpha^{-O(1)}$. Let $B,B',B''\subseteq G$ be regular Bohr sets of rank $d\leq \alpha^{-O(1)}$ such that $B''\subseteq B'_{c/d}$ and $A\subseteq B$ with relative density $\alpha$. If
    \[ \norm{ \mu_{A}\circ \mu_{A}}_{p(\mu_{k\cdot B'}\circ\mu_{k\cdot B'}\ast \mu_{k\cdot B''}\circ \mu_{k\cdot B''})} \geq \left(1+\epsilon\right) \mu(B)^{-1}\]
    then there is a regular Bohr set $B'''\subseteq B''$ of rank at most
    \[\rk(B''')\leq d+O_{\epsilon}(\lo{\alpha}^4p^2)\]
    and 
    \[\abs{B'''}\geq \exp(-O_{\epsilon}(d\lo{\alpha}+\lo{\alpha}^5p^2))\abs{B''}\]
    such that
    \[ \norm{ \mu_{B'''}*\mu_A }_\infty \geq (1+\epsilon/16)\mu(B)^{-1}. \]
\end{proposition}
\begin{proof}
As in the proof of \cite[Proposition 15]{BS}, there exist $A_1\subseteq k\cdot B'$ and $A_2\subseteq k\cdot B''-x$ such that, with $S=\{x\in A_1-A_2 : \mu_{A}\circ \mu_A(x)\geq (1+\epsilon/2)\mu(B)^{-1}\}$,
\[\langle \mu_{A_1}\circ \mu_{A_2},\ind{S}\rangle \geq 1-\epsilon/4\]
and
\[\min\brac{\mu_{k\cdot B'}(A_1),\mu_{k\cdot B''-x}(A_2)}\gg \alpha^{p+O_{\epsilon}(1)}.\]
We now apply Lemma~\ref{lemma-genimp} (with $k\cdot B'$ and $k\cdot B''$ playing the roles of $B$ and $B'$ respectively), noting that
\[\abs{S}\leq \abs{B'+B''}\leq 2\abs{B'},\]
and the conclusion follows.
\end{proof}

\end{document}